\documentclass[12pt]{article}
\usepackage{amsmath}
\usepackage{amssymb}
\usepackage{amsthm}
\usepackage{latexsym}
\usepackage{color,fancyhdr,lastpage,graphicx}
\usepackage{subfigure, epsfig, epsf}
\usepackage{xspace}
\usepackage{setspace}
\usepackage{bbm}

\graphicspath{{thesis/Figures/}}

\theoremstyle{plain}

\newtheorem{lemma}{Lemma}[section]

\newtheorem{theorem}[lemma]{Theorem}
\newtheorem{proposition}[lemma]{Proposition}

\theoremstyle{remark}

\def\bb{\begin{color}{blue}}
\def\bg{\begin{color}{green}}
\def\br{\begin{color}{red}}
\def\bbr{\begin{color}{brown}}
\def\eg{\end{color}}
\def\er{\end{color}}
\def\eb{\end{color}}
\def\ve{{\varepsilon}}
\def\le{\leqslant}

\def\ge{\geqslant}

\def\da{\downarrow}

\def\E{{\mathbb E}}

\def\R{{\mathbb R}}

\def\N{{\mathbb N}}
\def\T{{\mathbb T}}
\def\Z{{\mathbb Z}}

\def\a{{\alpha}}
\def\b{{\beta}}
\def\d{{\delta}}
\def\D{{\Delta}}
\def\t{{\tau}}
\def\k{{\kappa}}
\def\g{{\gamma}}

\def\s{{\sigma}}
\def\l{{\lambda}}

\def\z{{\zeta}}

\def\cB{{\cal B}}
\def\cK{{\cal K}}

\def\cL{{\cal L}}

\def\cE{{\cal E}}
\def\cF{{\cal F}}

\def\cV{{\cal V}}

\def\cM{{\cal M}}

\def\half{{\frac12}}
\def\q{\quad}
\def\fsn{{\lfloor s\rfloor_n}}
\def\csn{{\lceil s\rceil_n}}

\def\<{\langle}
\def\>{\rangle}
\def\ra{\rightarrow}
\def\ua{\uparrow}
\def\sse{\subseteq}
\def\sm{\setminus}

\renewcommand\epsilon{\ve}
\begin{document}
\bibliographystyle{plain}

%	FRONT MATTER
\begin{center}
\LARGE \textbf{Measure solutions for the Smoluchowski coagulation-diffusion equation}

\vspace{0.2in}

\large {\bfseries James Norris
\footnote{Statistical Laboratory, Centre for Mathematical Sciences,
  Wilberforce Road, Cambridge, CB3 0WB, UK }
}

\vspace{0.2in}
\small
\today

\end{center}
\begin{abstract}
A notion of measure solution is formulated for a coagulation-diffusion equation, 
which is the natural counterpart of Smoluchowski's coagulation equation in a spatially inhomogeneous setting.
Some general properties of such solutions are established.
Sufficient conditions are identified on the diffusivity, coagulation rates and initial data for existence, uniqueness and mass conservation of solutions.
These conditions impose no form of monotonicity on the coagulation kernel, which may depend on complex characteristics of the particles.
They also allow singular behaviour in both diffusivity and coagulation rates for small particles.
The general results apply to the Einstein--Smoluchowski model for colloidal particles suspended in a fluid.
\end{abstract}
\vspace{0.2in}

%	BODY OF ARTICLE

\section{Introduction}\label{I}
In a system of particles, subject both to diffusion and coagulation, we may think of each particle as characterized by a position $x\in\R^d$ and a type $y$
in some auxiliary space $E$.
The type of a particle might simply be its mass, in which case we would take $E=(0,\infty)$.  
In any case, we may suppose that the type of a particle governs both the diffusivity of its position and its tendency to coagulate with other particles.
We will consider deterministic models, 
where the state at a given time $t$ is a measure $\mu_t$ on $\R^d\times E$ describing the distribution of these characteristics of particles in the system.
Our aim is to show, subject to reasonable conditions on the initial state and on the diffusion and coagulation rates, that a natural differential
equation in measures, generalizing Smoluchowski's coagulation equation, has a unique solution, and to show a few properties of solutions, in particular conservation of mass.
The main novelty is our consideration of measure-valued solutions, where prior work has dealt with function-valued solutions, and in allowing the possibility that
coagulation rates may depend on complex characteristics of the particles. 
We are also able to handle some cases of unbounded diffusivity and unbounded coagulation rates for small particles, to which existing works do not apply, but which have some plausible physical relevance.
 
In the rest of this section, we introduce the needed mathematical framework and we set out our assumptions on the rates for coagulation and diffusion.
The notion of measure solution is defined in Section \ref{SOL} and some consequent properties are proved.
We prove in particular a new result showing that the property that solutions conserve total mass is independent of the notion of mass, 
meaning any quantity preserved in individual coagulation events.
Section \ref{HFT} reviews some facts on weak solutions of the heat equation in the context of multi-type diffusion.
Some alternative notions of measure solution are discussed in Section \ref{OS}. 
Related prior work on coagulation-diffusion is discussed in Section \ref{RW}. 
The main result is Theorem \ref{T1}, which gives conditions for existence, uniqueness and mass conservation of measure solutions.
Then Theorem \ref{IFU} shows that our measure solutions in fact give rise to function-valued solutions for suitable initial data.
Finally, Section \ref{BC} discusses an application to one case of physical interest.

Let $(E,{{\cE}})$ be a measurable space on which are given measurable functions $m:E\to (0,\infty)$ and $a: E\to (0,\infty)$,
and let $K$ be a finite kernel on $E\times E\times {{\cE}}$. 
Thus, $K(.,.,A)$ is a finite non-negative measurable function on $E\times E$, for all $A\in\cE$, and $K(y,y',.)$ is a measure on $\cE$ for all $y,y'\in E$. 
We assume that $K$ is symmetric in its first and second arguments. 
We think of $E$ as a set of {\it particle types} and interpret $m(y)$ as the {\it mass} and $a(y)$ as the {\it diffusivity} of a particle of type $y$.
We interpret $K(y,y',dz)$ as the {\em coagulation rate} for the event that a pair of particles of types $y,y'$ combines to become a particle of type $z$.  
We assume that $K$ is {\em mass-preserving}
$$
m = m(y) + m(y'),\q K(y, y',\cdot)\text{-a.e.}
$$
With few exceptions, existing work on coagulation is devoted to the case $E=(0,\infty)$ with $m(y)=y$.  
In this case every mass-preserving kernel has the form $k(y,y')\d_{y+y'}(dz)$ 
for some symmetric measurable function $k$ on $E\times E$, where $\d_y$ is the unit mass at $y$.  
We choose a more general framework to model physical processes where coagulation rates do not depend only on particle masses. 

Write $\cM$ for the set of finite measures $\mu$ on $(\R^d\times E,\cB(\R^d)\otimes\cE)$ 
whose first marginal $B\mapsto\mu(B\times E)$ is absolutely continuous with respect to Lebesgue measure on $\R^d$.
We use $\cM$ as the state-space for our dynamics, interpreting $\mu(B\times A)$ as the number of particles having position in $B$ and type in $A$.
We assume throughout that $(E,\cE)$ is a standard measurable space. 
This is not significantly restrictive for potential applications.
It ensures that, for all $\mu\in\cM$, there exists a kernel $\k$ on $\R^d\times\cE$, such that
$$
\mu(B\times A)=\int_{x\in B}\k(x,A)dx,\q B\in\cB(\R^d),\q A\in\cE.
$$
Moreover, if $\k$ and $\k'$ are both kernels for $\mu$, then the measures $\k(x,.)$ and $\k'(x,.)$ agree for almost all $x\in\R^d$.
We will abuse notation in writing $\mu(x,A)$ for $\k(x,A)$ where the choice of version is unimportant.

For suitable $\mu\in\cM$, we can determine a signed measure $K(\mu)$ on $\R^d\times E$ by
\begin{equation}\label{KMU}
K(\mu)(B\times A)=\frac12\int_{B\times E\times E\times E}\{1_A(z)-1_A(y)-1_A(y')\}K(y,y',dz)\mu(x,dy)\mu(x,dy')dx.
\end{equation}
The signed measure $K(\mu)$ will describe the rate of change in the state $\mu$ due to coagulation.

Let us say that $(\mu_t)_{t<T}$ is a {\em process} in $\cM$ if $\mu_t\in\cM$ for all $t$ and the map $t\mapsto\mu_t(B\times A):[0,T)\to[0,\infty)$ 
is measurable for all $B\in\cB(\R^d)$ and all $A\in\cE$.
We will find conditions under which the equation
\begin{equation}\label{CD}
\dot\mu_t=\tfrac12a\Delta\mu_t+K(\mu_t) 
\end{equation}
suitably interpreted, determines, for some $T\in (0,\infty]$, a unique process $(\mu_t)_{t<T}$ in $\cM$ starting from a given initial measure $\mu_0\in\cM$.
Here, $\Delta$ denotes the usual Laplacian on $\R^d$.
Then $(\mu_t)_{t<T}$ has the interpretation of an evolving cloud of particles in $\R^d$, of various types, 
where a particle of type $y$ diffuses in $\R^d$ at rate $a(y)$, and where two particles at the same
spatial location, of types $y$ and $y'$, coagulate to form a particle of type $z$ at rate $K(y,y',dz)$.
The class of measures $\cM$ is a natural one for this problem. 
In particular, the need to form the product $\mu_t(x,dy)\mu_t(x,dy')$ in $K(\mu_t)$ does not allow us to write an analogous equation for general measures on $\R^d\times E$.

\def\r{\>}
\renewcommand\l{\<}
We now reformulate the equation \eqref{CD} so that it makes sense for any process $(\mu_t)_{t<T}$ in $\cM$.
First, for $\mu\in\cM$, we can determine measures $K^\pm(\mu)$ on $\R^d\times E$ by
\begin{align}\label{K+}
K^+(\mu)(B\times A)
&=\frac12\int_{B\times E\times E}K(y,y',A)\mu(x,dy)\mu(x,dy')dx\\
\label{K-}
K^-(\mu)(B\times A)
&=\int_{B\times A\times E}K(y,y',E)\mu(x,dy)\mu(x,dy')dx.
\end{align}
Provided these measures are finite, we have $K^\pm(\mu)\in\cM$ and $K(\mu)=K^+(\mu)-K^-(\mu)$.
\def\j{
Given a kernel $\mu$ on $\R^d\times {\cE}$ and given $t\in(0,\infty)$, we can define\footnote{
We make several contructions of new kernels from old, all of which can be considered as special cases of the following. 
For $i=1,2,3$, let $(E_i,\cE_i)$ be a measurable space 
and let $\mu$ and $\nu$ be finite kernels on $E_1\times\cE_2$ and $(E_1\times E_2)\times\cE_3$ respectively.
Then there is a unique kernel $\mu\otimes\nu$ on $E_1\times(\cE_2\otimes\cE_3)$ such that, 
for all $x_1\in E_1$, all $A_2\in\cE_2$ and $A_3\in\cE_3$, we have
$$
(\mu\otimes\nu)(x_1,A_2\times A_3)=\int_{x_2\in A_2}\mu(x_1,dx_2)\nu(x_1,x_2,A_3).
$$
We write formally $(\mu\otimes\nu)(x_1,dx_2,dx_3)=\mu(x_1,dx_2)\nu(x_1,x_2,dx_3)$.
We then obtain a further kernel $\mu\circ\nu$ on $E_1\times\cE_3$ by integrating out over $E_2$
$$
(\mu\circ\nu)(x_1,A_3)=(\mu\otimes\nu)(x_1,E_2,A_3).
$$
}
new kernels $K^\pm (\mu)$ and $P_t\mu$ on $\R^d \times {\cE}$ by
\begin{align*}
K^\pm (\mu) (x,dy) & = K^\pm (\mu (x,\cdot ))(dy), \\
P_t\mu (x,dy) & = \int_{\R^{d}} p(a(y)t, x,x') \mu (x',dy) dx'
\end{align*}
}
Next, for $t>0$, define $P_t\mu\in\cM$ by
$$
P_t\mu(B\times A)=\int_{\R^d\times\R^d\times E}\mu(dx,dy)p(a(y)t,x,x')1_B(x')dx'
$$
where $p(t,x,x')= (2\pi t)^{-d/2}\exp\{-|x-x'|^2/2t\}$. 
Then \eqref{CD} is formally equivalent to the following equation in measures on $\R^d\times E$
\begin{equation*}\label{WCDO}
\mu_t+\int^t_0P_{t-s}K^-(\mu_s)ds=P_t\mu_0+\int^t_0P_{t-s}K^+(\mu_s)ds,\q t\in(0,T).
\end{equation*}
By writing in this mild form and by rearranging the non-linear terms, we do not need any assumptions of regularity or integrability
to make sense of the equation.

We now introduce our main assumptions on the diffusivity $a$ and the coagulation kernel $K$.  
First, here is some terminology.  
Say that a function $f$ on $E$ is {\it locally bounded} if it is bounded on $m^{-1}(B)$ for all compact sets $B \sse (0,\infty)$.  
Say that $f$ is $K$-{\it decreasing} if, for all $y, y'\in E$, we have $f \leq f(y)$, $K(y,y',.)$-a.e.  
Say that $f$ is $K$-{\it subadditive} if, for all $y, y'\in E$, we have $f \leq f(y)+f(y')$, $K(y,y',.)$-a.e.  
When $E=(0,\infty)$ and $m(y) = y$, these notions coincide with the usual notions of locally-bounded, non-increasing and subadditive function on $(0,\infty)$.
We choose a continuous function $\phi :(0,\infty)\to (0,\infty)$ such that $\phi(\lambda m)\le\lambda\phi(m)$ for all $\lambda\ge1$ and all $m\in(0,\infty)$.
Set $w(y)=a(y)^{d/2}\phi(m(y))$.  
We will assume throughout:
\begin{align}
&\text{$a$ is locally bounded and $K$-decreasing},\q\text{$a^{-1}$ is locally bounded},\notag\\
&\text{$w$ is uniformly positive},\q K(y,y',E)\leq w(y)w(y'),\q y,y' \in E.
\label{A} 
\end{align}
The assumption that diffusivity decreases on coagulation is physically reasonable. 
The upper bound on $K$ will dictate our choice of $\phi$ and then the sublinearity condition on 
$\phi$ will restrict us to cases where coagulation rates do not increase too rapidly with increasing particle mass.  

\section{Solutions and their properties}\label{SOL}
Let $(\mu_t)_{t<T}$ be a process in $\cM$.
We say that $(\mu_t)_{t<T}$ is a {\it solution}\footnote{This would often be called a mild solution.} to the coagulation-diffusion equation \eqref{CD} if 
\begin{equation}\label{WCD}
\mu_t+\int^t_0P_{t-s}K^-(\mu_s)ds=P_t\mu_0+\int^t_0P_{t-s}K^+(\mu_s)ds,\q t\in(0,T)
\end{equation}
and the following integrability conditions hold: 
\begin{equation}\label{M0}
\int_{\R^d\times E}w(y)\mu_0(dx,dy)<\infty,\q
\mu_0(dx,dy)\le dx\otimes\mu_0^*(dy)
\end{equation}
for some measure $\mu_0^*$ on $E$ for which $w$ is integrable, and
\begin{equation}\label{W}
\int^t_0\int_Ew^R(y)P_{t-s}K^+(\mu_s)(x,dy)ds<\infty,\q\text{a.a. }x\in\R^d,\q R\in(0,\infty),\q t<T
\end{equation}
where
$$
\phi^R(m)=m1_{m\leq  R},\q w^R(y)=a(y)^{d/2}\phi^R(m(y)).
$$

The final condition \eqref{W} is needed to make \eqref{WCD} informative and so offer a chance of proving uniqueness.
Without \eqref{W}, equation \eqref{WCD} might just say `$\infty=\infty$'.
A simpler but stronger condition is to require that that left side of \eqref{W} is integrable over $\R^d$, which can be written
$$
\int^t_0\int_{\R^d\times E\times E\times E}w(z)1_{\{m(z)\le R\}}K(y,y',dz)\mu_s(x,dy)\mu_s(x,dy')dxds<\infty, \q R\in(0,\infty),\q t<T.
$$
Here we used Fubini and integrated out the density function $p(a(y)t,x,x')$.
%This is strengthened again if we replace $w^R$ by $w$.

We say that $(\mu_t)_{t<T}$ is a {\it strong solution}\footnote{This terminology is non-standard -- `strong' here refers mainly to the finite-second-moment-type condition \eqref{S2}, not to 
any additional smoothness.} if \eqref{WCD} and \eqref{M0} hold, together with
\begin{equation}\label{S1}
\sup_{s \leq t} \|\< w, \mu_s \> \|_1 < \infty, \quad t< T,
\end{equation}
and
\begin{equation}\label{S2}
\int^t_0 \|\< w^2, \mu_s\> \|_\infty ds < \infty, \quad t<T. 
\end{equation}
Here $\<w,\mu_s\>$ is the measurable function on $\R^d$ obtained by integrating $w$ with respect to the kernel $\mu_s(x,.)$ over $E$, and $\|.\|_p$ is the $L^p$-norm on $\R^d$.
We will see shortly that these conditions imply (\ref{W}), so a strong solution is indeed a solution.

%In the remainder of this section we first obtain some {\it a priori}
%bounds on solutions, then an existence and uniqueness result, Theorem
%\ref{T1}. Finally, we show regularity of solutions under appropriate
%hypotheses. In particular, we show, when $E=(0,\infty)$ and for suitable 
%initial data, that our
%kernel-valued solutions are absolutely continuous on $E$ and $C^2$ on $\R^d$
%and that \eqref{CD} holds in a direct sense.

%It is shown below\footnote{or will be} that under suitable hypotheses 
%our kernel-valued solutions have enough regularity to be considered as
%solutions to \eqref{CD} in a direct sense.

Note that, for all $t>0$ and $x,x'\in \R^d$, the function
$$
w^{t,x,x'}(y) = w(y) p(a(y)t, x_,x') = \phi (m(y)) (2\pi t)^{-d/2} e^{-|x-x'|^2/2a(y)t}
$$
is $K$-subadditive\footnote{Here we strongly use the explicit Gaussian form of the transition density, 
which appears to rule out an extension of our approach to spatially dependent diffusion. 
On the other hand, a similar inequality does hold for Brownian motion on the torus, by replacing $x'$ by $x'+n$ and summing over $n\in\Z^d$.}
and, for any measure $\mu\in\cE$, the following integral is well-defined and non-positive:
\begin{equation}\label{F}
 {\half} \int_{\R^{d}\times E\times E \times E}
\{w^{t,x,x'} (z)-w^{t,x,x'} (y)-w^{t,x,x'} (y')\}
K(y,y',dz) \mu(x',dy)\mu(x',dy')dx'.
\end{equation}
We will denote this integral by $\<w,P_tK(\mu)\>(x)$, noting that
$$
\<w,P_tK(\mu)\>=\<w,P_tK^+(\mu)\>-\<w,P_tK^-(\mu)\>
$$
whenever the first term on the right is finite.

\begin{proposition}\label{P1}
Suppose $(\mu_t)_{t<T}$ is a solution to \eqref{CD}. 
Then, for all $t<T$, we have
\begin{equation}\label{B1}
\l w, \mu_t\r \leq \l w, P_t \mu_0\r + \int^t_0 \l w, P_{t-s} K(\mu_s)\r ds\q\text{a.e.}
\end{equation}
and hence
$$
\|\<w,\mu_t\>\|_1\le\|\<w,\mu_0\>\|_1<\infty,\q
\|\<w,\mu_t\>\|_\infty\le\<w,\mu_0^*\><\infty.
$$
In particular, \eqref{S1} holds and
\begin{equation}\label{M1}
\sup_{s \leq t <T} \|\l 1, P_{t-s} K^+ (\mu_s)\r \|_1 \ < \infty.
\end{equation}
\end{proposition}

\begin{proof}
Define $w_n(y)=a(y)^{d/2}\phi_n(m(y))$, where $\phi_n$ is the sublinear function
$$
\phi_n (m) = (m1_{m \leq n^{-1}}) n \phi( n^{-1}) + 1_{n^{-1} < m \leq n} \phi(m).
$$
Then $w_n\le w^n$ and $w_n\uparrow w$ on $E$ as $n\to\infty$.  
By \eqref{W} we know that, for all $n\in\N$, all $t<T$, we have
$$
\int^t_0 \l w_n, P_{t-s} K^+ (\mu_s)\r ds < \infty\q\text{a.e.}
$$
so we can multiply \eqref{WCD} by $w_n$, integrate over $E$ and rearrange to obtain
$$
\l w_n, \mu_t \r =
\l w_n, P_t \mu_0\r +
\int^t_0 \l w_n, P_{t-s} K(\mu_s) \r ds\q\text{a.e.}
$$
Now write $\l w_n, P_{t-s} K(\mu_s)\r$ as an integral over $\R^d\times E \times E \times E$ as in \eqref{F} 
and pass to the limit using Fatou's lemma to obtain \eqref{B1}. 
From \eqref{B1} we deduce that $\l w, \mu_t\r \leq \l w, P_t\mu_0\r$ almost everywhere, 
so $\|\l w,\mu_t\r\|_1 \leq \|\l w, P_t \mu_0\r \|_1 = \|\l w, \mu_0\r \|_1<\infty$ 
and $\|\l w,\mu_t\r\|_\infty \leq \|\l w, P_t \mu_0\r \|_\infty \le\<w,\mu_0^*\><\infty$ for all $t<T$.
Then \eqref{M1} follows from
\begin{align*}
\|  \l 1, &P_{t-s} K^+ (\mu_s)\r \|_1 
= \|\l 1, K^+ (\mu_s)\r\|_1
=\half \| \l 1, K^- (\mu_s)\r \|_1\\
& = \int_{\R^{d}\times E \times E} K(y,y',E) \mu_s (x,dy) \mu_s (x,dy') \, dx\\
&\leq  \| \l w, \mu_s\r^2\|_1 
\leq  \| \l w, P_s\mu_0\r^2\|_1 
\le\|\l w, \mu_0 \r \|_1\l w, \mu^*_0 \r < \infty.
\end{align*}
\end{proof}

\begin{proposition}
Suppose $(\mu_t)_{t<T}$ is a strong solution to \eqref{CD}. 
Then, for all $t<T$, we have
\begin{equation}
\int^t_0 \|\l w, P_{t-s} K^+ (\mu_s)\r \|_1ds < \infty\label{MW}
\end{equation}
so \eqref{W} holds and
\begin{equation}
\l w,\mu_t\r = \l w, P_t \mu_0 \r +
\int^t_0 \l w, P_{t-s} K(\mu_s)\r ds\q\text{a.e}.\label{BE}
\end{equation}
\end{proposition}

\begin{proof}
We have
\begin{align*}
\| \l w, P_{t-s} &K^+ (\mu_s)\r \|_1 
\leq \|\l w, P_{t-s} K^- (\mu_s)\r \|_1
=\|\l w, K^- (\mu_s)\r \|_1\\
&=\int_{\R^{d}\times E \times E} w(y) K(y,y',E)\mu_s(x,dy)\mu_s(x,dy')dx\\
&\leq\|\l w, \mu_s \r \l w^2, \mu_s \r \|_1 
\leq \| \l w, \mu_s \r\|_1 \, \|\l w^2, \mu_s\r \|_\infty.
\end{align*}
So \eqref{MW} follows from \eqref{S1} and \eqref{S2}.

Now \eqref{W} holds because, for all $R<\infty$, for some $\ve>0$, $\ve m1_{m \leq R} \leq \phi (m)$ for all $m\in (0,\infty)$.  
On multiplying \eqref{WCD} by $w$ and integrating over $E$, all terms are integrable over $\R^d$, hence finite almost everywhere.  
On rearranging we obtain \eqref{BE}.
\end{proof}

\begin{proposition}
Let $(\mu_t)_{t<T}$ be a solution to \eqref{CD}.  Then 
$\|\l m, \mu_t\r \|_1$ is non-increasing in $t$.
\end{proposition}

\begin{proof} Fix $R<\infty$. Multiply \eqref{WCD} by $m1_{m
\leq R}$ and integrate over $\R^d \times E$ to obtain
\begin{align*}
\|\l m1_{m\leq R},& \, \mu_t \r \|_1 + \int^t_0 \| \l m1_{m\leq R}, \ K^-
(\mu_s)\r \|_1 \, ds \\
&= \| \l m1_{m\leq R}, \, \mu_0 \r \|_1 + \int^t_0 \| \l m1_{m \leq R},
K^+ (\mu_s) \r \|_1 \, ds,
\end{align*}
with all terms finite by Proposition \ref{P1}.  
Since $m1_{m \leq R}$ is $K$-subadditive,
$$
\<m1_{m\le R},K^+(\mu_s)\>\le\<m1_{m\le R},K^-(\mu_s)\>
$$
so
$$
\|\l m1_{m\leq R}, \ \mu_t \r \|_1 \leq \| \l m1_{m\leq R}, \mu_0 \r \|_1,
$$
and the claim follows by monotone convergence.
\end{proof}

Let us call a measurable function $n:E\to(0,\infty)$ a {\em mass function for $K$} if
$$
n=n(y)+n(y'),\q K(y,y',\cdot)\text{-a.e.},\q y,y'\in E.
$$
In particular, $m$ is a mass function, and has been given a special role in the discussion above.
However, it is possible that $K$ may have more than one conserved quantity. For example, the type of a particle
may determine the number of initial particles present, which is then a mass function for $K$.
The proof just
given shows that $\|\l n, \mu_t\r \|_1$ is non-increasing in $t$ for any mass function $n$.
Let us say that a solution $(\mu_t)_{t<T}$ to \eqref{CD} {\it conservative} if $\|\l m, \mu_t\r \|_1=\|\l m, \mu_0\r \|_1<\infty$
for all $t<T$. 

\begin{proposition}
Let $(\mu_t)_{t<T}$ be a solution to \eqref{CD} with $\|\l m, \mu_0\r \|_1<\infty$.
Let $n$ be a mass function for $K$ with $\|\l n, \mu_0\r \|_1<\infty$.
Then $(\mu_t)_{t<T}$ is conservative if and only if it is {\em $n$-conservative}, that is to say if
$\|\l n, \mu_t\r \|_1=\|\l n, \mu_0\r \|_1$ for all $t<T$.
\end{proposition}
\begin{proof} 
Suppose that $(\mu_t)_{t<T}$ is conservative and that $f$ is a non-negative measurable function on $E$
such that $f\le m$ and 
$$
f\ge f(y)+f(y'),\q K(y,y',\cdot)\text{-a.e.},\q y,y'\in E.
$$
Fix $R\in[0,\infty)$ and set $m_R(y)=m(y)1_{m(y)\le R}$ and $f_R(y)=f(y)1_{m(y)\le R}$. 
Write $\{f\}(y,y',z)$ for $f(y)+f(y')-f(z)$. It is straightforward to check that
$\{f_R\}(y,y',z)\le\{m_R\}(y,y',z)$ for $K(y,y',.)$-almost all $z$, for all $y,y'$.
Moreover $m_R$ and $f_R$ are both bounded, so
\begin{align*}
&\|\l f_R,\mu_0\r\|_1-\|\l f_R,\mu_t\r\|_1\\
&\q\q=\frac12\int_0^t\int_{\R^d\times E\times E\times E}\{f_R\}(y,y',z)K(y,y',dz)\mu_s(x,dy)\mu_s(x,dy')dxds\\
&\q\q\le\frac12\int_0^t\int_{\R^d\times E\times E\times E}\{m_R\}(y,y',z)K(y,y',dz)\mu_s(x,dy)\mu_s(x,dy')dxds\\
&\q\q=\|\l m_R,\mu_0\r\|_1-\|\l m_R,\mu_t\r\|_1.
\end{align*}
On letting $R\to\infty$, we see that $\|\<f,\mu_t\>\|_1\ge\|\<f,\mu_0\>\|_1$.

For each $N\in\N$, the preceding argument may be applied with $f=(n/N)\wedge m$ to show that $\|\<n\wedge(Nm),\mu_t\>\|_1\ge\|\<n\wedge(Nm),\mu_0\>\|_1$.
Since $m$ is positive, we obtain $\|\<n,\mu_t\>\|_1\ge\|\<n,\mu_0\>\|_1$ on letting $N\to\infty$.
Hence $(\mu_t)_{t<T}$ is $n$-conservative.
The same argument shows that $n$-conservativity implies conservativity.
\def\j{
For general $n$, by enriching the type space if necessary,
we may assume\footnote{Write $\tilde E$ for the subset of those $(y,\t)\in E\times\T(E)$ with $m(y)=m(\t)$ and $n(y)=n(\t)$
where we have extended $m$ to $\T(E)$ using $m(\{\t_1,\t_2\})=m(\t_1)+m(\t_2)$, and similarly $n$. Put $\mu_0$ on the diagonal 
in $E\times\T_1(E)=E\times E$.
Replace the coagulation event $(y,y')\ra z$ by $((y,\t),(y',\t'))\ra(z,\{\t,\t'\})$. Define $m$ and $n$ on $\tilde E$ in the obvious way.
Given any positive measurable function $f$ on $E=\T_1(E)$, we can extend $f$ to a mass function on $\T(E)$ using $f(\{\t_1,\t_2\})=f(\t_1)+f(\t_2)$.
Then define $f(y,\t)=f(\t)$, a mass function on $\tilde E$. Apply this construction with $f=Nm\wedge n$ to obtain $m_N$.
Then $m_N$ has the claimed properties. Given a solution $(\mu_t)_{t<T}$ to \eqref{CD}, it is straightforward to construct a
solution $(\tilde\mu_t)_{t<T}$ to the lifted coagulation equation in $\tilde E$, which is conservative if and only if $(\mu_t)_{t<T}$ is so,
and is $n$-conservative if and only if $(\mu_t)_{t<T}$ is so.}
there exists a sequence $(\tilde m_N:n\in\N)$ of mass functions for $K$
such that $\tilde m_N\le Nm$ for all $N$ and $\tilde m_N\ua n$ as $N\to\infty$.
Then $(\mu_t)_{t<T}$ is $\tilde m_N$-conservative for all $N$ by the preceding
argument, and it follows by monotone convergence that $(\mu_t)_{t<T}$ is $n$-conservative.
}
\end{proof}

\section{Heat flow with types}\label{HFT}
We discuss briefly the propagators associated to the time-dependent and type-dependent differential operator on $\R^d$
$$
\cL_t=\tfrac12a(y)\Delta+g_t(.,y).
$$
Here, for simplicity, we will assume that the diffusivity $a$ is measurable and satisfies
\begin{equation}\label{AU}
\inf_{y\in E}a(y)>0,\q\sup_{y\in E}a(y)<\infty
\end{equation} 
and that $g=(g_t)_{t\ge0}$ is a process of measurable functions on $\R^d\times E$ such that
\begin{equation}\label{GU}
\|g\|_\infty:=\sup_{t\ge0,\,y\in E}\|g_t(.,y)\|_\infty<\infty.
\end{equation} 
A generalization to the case where these conditions hold on $E_n$ for some measurable sets $E_n\ua E$ is obvious.
The lines of the discussion are standard, but it will serve to introduce notation and to check its applicability in the time-dependent and type-dependent case.
For $0\le s<t$, for $x,x'\in\R^d$ and for $a\in(0,\infty)$, write $\b_a^{s,x;t,x'}$ for the Borel measure on the set of continuous paths $C([s,t],\R^d)$ 
which is the law of a Brownian bridge of diffusivity $a$, starting from $x$ at time $s$ and ending at $x'$ at time $t$.
Note that
\begin{equation}\label{BBM}
\b_a^{s,x;t,x'}=\b_1^{0,0;1,0}\circ\phi^{-1}
\end{equation}
where $\phi:C([0,1],\R^d)\to C([s,t],\R^d)$ is given by
$$
\phi(w)((1-\t)s+\t t)=x+\sqrt{a}w(\t)+\t(y-x),\q 0\le\t\le1.
$$
Define
$$
p_y(s,x;t,x')=p(a(y)(t-s),x,x')\pi_y(s,x;t,x')
$$
where
$$
\pi_y(s,x;t,x')=\int_{C([s,t],\R^d)}\exp\left\{\int_s^tg_\t(w_\t,y)d\t\right\}\b_{a(y)}^{s,x;t,x'}(dw).
$$
Note that $\pi_y(s,x;t,x')\le e^{(t-s)\|g\|_\infty}$.
The function $p$ is jointly measurable in all variables. This can be seen using \eqref{BBM}.
Define the propagators $(P_{st}:0\le s<t)$ on bounded measurable functions on $\R^d\times E$ and $(P_{ts}:0\le s<t)$ on signed measures on $\R^d\times E$ of finite total variation by
$$
P_{st}f(x,y)=\int_{x'\in\R^d}p_y(s,x;t,x')f(x',y)dx',\q
P_{ts}\mu(dx',dy)=\int_{x\in\R^d}\mu(dx,dy)p_y(s,x;t,x')dx'.
$$
It will be convenient to agree also that $P_{ss}f=f$ and $P_{ss}\mu=\mu$ for all $s$.
Write $(f,\mu)$ for the integral $\int_{\R^d\times E}f(x,y)\mu(dx,dy)$.
By Fubini, we have $(P_{st}f,\mu)=(f,P_{ts}\mu)$ for all $f$ and $\mu$. 
By the Markov property of Brownian motion, we have $P_{st}\circ P_{tu}=P_{su}$ for $0\le s<t<u$.
Also, $P_{st}f(x,y)=\E(f(B_t,y)Z_t)$, where $(B_t)_{t\ge s}$ is a Brownian motion in $\R^d$ of diffusivity $a(y)$, starting from $x$ at time $s$, 
and where $Z_t=\exp\int_s^tg_\t(B_\t,y)d\t$.
Write $\cF$ for the set of all bounded measurable functions on $\R^d\times E$ which are twice continuously differentiable along $\R^d$ 
with bounded first and second derivatives.
For $f\in\cF$, by It\^o's formula,
$$
d(f(B_t,y)Z_t)=\nabla f(B_t,y)Z_tdB_t+\left(\tfrac12a(y)\Delta f(B_t,y)+g_t(B_t,y)f(B_t,y)\right)Z_tdt
$$
so we obtain, on taking expectations,
\begin{equation}\label{PSF}
P_{st}f(x,y)=f(x,y)+\int_s^tP_{s\t}\cL_\t f(x,y)d\t.
\end{equation}
We will write $\|\mu\|_1$ for the total variation of a signed measure $\mu$ on $\R^d\times E$.

\begin{proposition}\label{WSS}
Assume that $a$ satisfies \eqref{AU} and $g$ satisfies \eqref{GU}.
Let $\mu_0$ be a signed measure on $\R^d\times E$ and let $(\a_t)_{t\le T}$ be a process of such signed measures.
Assume that
\begin{equation}\label{DCK}
\|\mu_0\|_1+\int_0^T\|\a_t\|_1dt<\infty
\end{equation}
Define a process of signed measures on $\R^d\times E$ by
\begin{equation}\label{MCK}
\mu_t=P_{t0}\mu_0+\int_0^tP_{ts}\a_sds,\q t\le T.
\end{equation}
Then
\begin{equation}\label{FCK}
\sup_{t\le T}\|\mu_t\|_1<\infty
\end{equation}
and
\begin{equation}\label{WCK}
(f,\mu_t)=(f,\mu_0)+\int_0^t(\cL_sf,\mu_s)ds+\int_0^t(f,\a_s)ds,\q f\in\cF,\q t\le T.
\end{equation}
On the other hand, in the case $g=0$, $(\mu_t)_{t\le T}$ is the only process of signed measures on $\R^d\times E$ satisfying \eqref{FCK} such that \eqref{WCK} holds.
Hence $(\mu_t)_{t\le T}$ satisfies
$$
\mu_t=P_t\mu_0+\int_0^tP_{t-s}(g_s\mu_s+\a_s)ds,\q t\le T.
$$
\end{proposition}
\begin{proof}
Suppose that $(\mu_t)_{t\le T}$ is given by \eqref{MCK}.
Then, for $t\le T$,
$$
\|\mu_t\|_1\le e^{t\|g\|_\infty}\|\mu_0\|_1+\int_0^te^{(t-s)\|g\|_\infty}\|\a_s\|_1ds\le e^{T\|g\|_\infty}\left(\|\mu_0\|_1+\int_0^T\|\a_s\|_1ds\right)
$$
so \eqref{FCK} holds.
Multiply \eqref{MCK} by $f\in\cF$ and integrate over $\R^d\times E$ to obtain
\begin{equation}\label{FPT}
(f,\mu_t)=(f,P_{t0}\mu_0)+\int_0^t(f,P_{ts}\a_s)ds=(P_{0t}f,\mu_0)+\int_0^t(P_{st}f,\a_s)ds.
\end{equation}
Now substitute for $P_{0t}f$ and $P_{st}f$ using \eqref{PSF} and reorder integrals using Fubini to obtain \eqref{WCK}. 

Suppose on the other hand that $(\mu_t)_{t\le T}$ satisfies \eqref{FCK} and \eqref{WCK}.
Define, for $t\le T$,
$$
\nu_t=\mu_t-P_{t0}\mu_0-\int_0^tP_{ts}\a_sds.
$$
Then, for all $f\in\cF$,
$$
(f,\nu_t)=(f,\mu_t)-(P_{0t}f,\mu_0)-\int_0^t(P_{st}f,\a_s)ds.
$$
Assume\footnote{The argument can be pursued also under suitable regularity conditions on $g$, at the cost of further elaboration, so we will
retain the more general notation.} now that $g=0$. 
Then $\cF$ is stable under $P_{st}$ and, for all $f\in\cF$, we have
\begin{equation}\label{PFG}
P_{st}f(x,y)=f(x,y)+\int_s^t\cL_rP_{rt}f(x,y)dr.
\end{equation}
Fix $n\in\N$ and set $\fsn=(t/n)\lfloor ns/t\rfloor$ and $\csn=(t/n)\lceil ns/t\rceil$.
Then
$$
(f,\mu_t)-(P_{0t}f,\mu_0)=\sum_{k=0}^{n-1}(P_{s_{k+1}t}f,\mu_{s_k+1})-(P_{s_kt}f,\mu_{s_k}).
$$
We use \eqref{WCK} and \eqref{PFG} to rewrite each term in the sum. Hence we obtain
$$
(f,\nu_t)=\int_0^t\{(\cL_s(P_{\csn t}-P_{st})f,\mu_s)+(\cL_sP_{st}f,\mu_s-\mu_\fsn)+((P_{\csn t}-P_{st})f,\a_s)\}ds.
$$
Now let $n\to\infty$, using \eqref{FCK} for estimates. 
Restrict to test-functions $f$ which are four times continuously differentiable, with all derivatives bounded uniformly on $\R^d\times E$.
Then \eqref{WCK} holds with $f$ replaced by $\cL_sP_{st}f$, which allows us to estimate the second term on the right.
We conclude that $(f,\nu_t)=0$.
Hence $\nu_t=0$ for all $t\le T$, so \eqref{MCK} holds.

For the final assertion, note that \eqref{WCK} can be written as
$$
(f,\mu_t)=(f,\mu_0)+\int_0^t(\tfrac12a(y)\Delta f,\mu_s)ds+\int_0^t(f,g_s\mu_s+\a_s)ds,\q f\in\cF,\q t\le T
$$
and apply the uniqueness result for the case $g=0$ with $\a_t$ replaced by $g_t\mu_t+\a_t$.
\end{proof}

\section{Other notions of solution}\label{OS}
We now discuss some alternative notions of solution for the coagulation-diffusion equation \eqref{CD} and establish
relations with the one already introduced. 
For the first two, we restrict the solution class for the process $(\mu_t)_{t<T}$ by the condition
\begin{equation}\label{MCD}
\sup_{s\le t}\left(\|\<w,\mu_s\>\|_1+\|\<w,\mu_s\>\|_\infty\right)<\infty, \q t<T.
\end{equation}
We showed in Proposition \ref{P1} that \eqref{MCD} is a natural property of solutions.

First we discuss a notion of weak solution. 
We take as our class of test-functions $\cF$ the set of all bounded measurable functions $f:\R^d\times E\to\R$, 
supported on $\R^d\times m^{-1}(B)$, for some compact set $B\sse(0,\infty)$, 
and such that $f(.,y)$ is twice continuously differentiable on $\R^d$ for all $y\in E$, 
with first and second derivatives bounded on $\R^d\times E$. 
Say that a process $(\mu_t)_{t<T}$ in $\cM$ is a {\em weak solution} to \eqref{CD} if \eqref{MCD} holds and
\begin{equation}\label{WS}
(f,\mu_t)=(f,\mu_0)+\int_0^t(\tfrac12a\D f,\mu_s)ds+\int_0^t(f,K(\mu_s))ds,\q f\in\cF,\q t<T.
\end{equation}
Note that \eqref{MCD} ensures that $K(\mu_s)$ is a well-defined signed measure and indeed, since
$w$ is assumed to be uniformly positive, all integrals in \eqref{WS} are well-defined and finite.

We turn to the second alternative notion of solution.
Note that, for a process $(\mu_t)_{t<T}$ in $\cM$, we have $K^-(\mu_t)=c_t\mu_t$, where
$$
c_t(x,y)=\int_EK(y,y',E)\mu_t(x,dy'),\q t\in[0,T),\q x\in\R^d,\q y\in E.
$$
Note also that $c_t(x,y)\le w(y)\<w,\mu_t\>(x)$.
For $n\in\N$, set $E_n=\{y\in E:n^{-1}\le m(y)\le n\}$.
Under conditions \eqref{A} and \eqref{MCD}, for all $n$, the diffusivity $a$ is uniformly positive and bounded on $E_n$ and 
$$
\sup_{s\le t,\,y\in E_n}\|c_s(.,y)\|_\infty<\infty,\q t<T.
$$
Write $P^\mu$ for the propagators associated to the operator $\frac12a(y)\Delta-c_t(.,y)$ on $\R^d$, corresponding to the choice $g=-c$ in the preceding section.
Say that $(\mu_t)_{t<T}$ is a {\em Markov solution} of  \eqref{CD} if \eqref{MCD} holds and
\begin{equation}\label{MS}
\mu_t=P^\mu_{t0}\mu_0+\int_0^tP_{ts}^\mu K^+(\mu_s)ds,\q t\in(0,T).
\end{equation}
We use the name Markov mainly to distinguish this notion of solution from others, 
but also because \eqref{MS} is the forward equation for the distribution of the associated non-linear Markov process.

\begin{proposition}\label{MSS}
Let $(\mu_t)_{t<T}$ be a process in $\cM$ and assume that \eqref{MCD} holds.
Then $(\mu_t)_{t<T}$ is a solution to \eqref{CD} if and only if it is a weak solution.
Moreover, $(\mu_t)_{t<T}$ is a solution to \eqref{CD} whenever it is a Markov solution.
\end{proposition}
\begin{proof}
For the first assertion, take $g_t=0$ and $\a_t=K^+(\mu_t)-c_t\mu_t$ in Proposition \ref{WSS}.
The second assertion is obtained similarly, by taking $g_t=-c_t$ and $\a_t=K^+(\mu_t)$.
\end{proof}

We now prove an {\em a priori} regularity property of solutions in the position variable.
Write $\cK$ for the set of kernels $\k$ on $\R^d\times\cE$ such that $\int_{\R^d}\k(x,E)dx<\infty$.
Given a kernel $\k$ on $\R^d\times\cE$ and given $t>0$, define new kernels $K^\pm(\k)$ and $P_t\k$ on $\R^d\times\cE$ by
\begin{align*}
K^+(\k)(x,A)
&=\frac12\int_{E\times E}K(y,y',A)\k(x,dy)\k(x,dy')\\
K^-(\k)(x,A)
&=\int_{A\times E}K(y,y',E)\k(x,dy)\k(x,dy')\\
P_t\k(x,A)
&=\int_{\R^d\times A}p(a(y)t,x,x')\k(x',dy)dx'.
\end{align*}
Note that, if $\k$ is a version of the density for $\mu\in\cM$, then
$K^\pm(\k)$ and $P_t\k$ are versions of the densities for $K^\pm(\mu)$ and $P_t\mu$.
Say that $(\k_t)_{t<T}$ is a process in $\cK$ if $\k_t\in\cK$ for all $t$ and the map $t\mapsto\k_t:[0,T)\to\cK$ is measurable.
Equivalently, the map $(t,x,A)\mapsto\k_t(x,A):[0,T)\times\R^d\times\cE$ is a kernel such that $\int_{\R^d}\k_t(x,E)dx<\infty$ for all $t<T$.

We say that a process $(\k_t)_{t<T}$ in $\cK$ is a {\em precise solution} to \eqref{CD} if 
\begin{equation}\label{PCD}
\k_t+\int^t_0P_{t-s}K^-(\k_s)ds=P_t\k_0+\int^t_0P_{t-s}K^+(\k_s)ds,\q t\in(0,T)
\end{equation}
and the following integrability conditions hold: 
\begin{equation}\label{PM0}
\int_{\R^d\times E}w(y)\k_0(x,dy)dx<\infty,\q
\k_0(x,dy)dx\le dx\otimes\mu_0^*(dy)
\end{equation}
and
\begin{equation}\label{PW}
\int^t_0\int_Ew^R(y)P_{t-s}K^+(\k_s)(x,dy)ds<\infty,\q\text{a.a. }x\in\R^d,\q R\in(0,\infty),\q t\in(0,T)
\end{equation}
where $\mu^*_0$ and $w^R$ are as in \eqref{M0} and \eqref{W}.
Note that \eqref{PCD} is an equality of kernels, not measures, and that no exceptional sets in $\R^d$ are allowed.
It is clear that, if $(\k_t)_{t<T}$ is a precise solution to \eqref{CD}, then there is a unique solution $(\mu_t)_{t<T}$ to \eqref{CD} such that
\begin{equation}\label{PM1}
\mu_t(B\times A)=\int_{x\in B}\k_t(x,A)dx,\q B\in\cB(\R^d),\q A\in\cE.
\end{equation}

\begin{proposition}
Let $(\mu_t)_{t<T}$ be a solution to \eqref{CD}. 
Choose a kernel $\k_0$ for $\mu_0$. 
Then there is a unique\footnote{In fact, for $t\in(0,T)$, $\k_t$ does not depend on the choice of $\k_0$.} precise solution $(\k_t)_{t<T}$ to \eqref{CD} starting from $\k_0$ such that \eqref{PM1} holds.
\end{proposition}
\begin{proof}
Define, for $t\in(0,T)$, kernels $\k_t^0$ and $\k_t^\pm$ on $\R^d\times\cE$ by
$$
\k_t^0(x,A)=(P_t\mu_0)(x,A),\q
\k_t^\pm(x,A)=\int_0^t(P_{t-s}K^\pm(\mu_s))(x,A)ds.
$$
Fix $A\in\cE$ such that $a(y)\in[a_1,a_2]$ for all $y\in A$, for some $0<a_1<a_2<\infty$.
Then 
$$
p(a(y)t,x,x')\le(a_2/a_1)^{d/2}p(a_2t,x,x')
$$
for all $y\in A$.
We use this inequality, together with the bound $\|\<w,\mu_t\>\|_\infty\le\<w,\mu_0^*\>$ from Proposition \ref{P1}, to see that, for all $x\in\R^d$,
\begin{align*}
\k_t^+(x,A)
&=\int_0^t\int_{\R^d\times E\times E\times E}1_A(z)p((t-s)a(z),x,x')K(y,y',dz)\mu_s(x',dy)\mu_s(x',dy')dx'ds\\
&\le(a_2/a_1)^{d/2}\int_0^t\int_{\R^d}p(a_2(t-s),x,x')\<w,\mu_s\>^2(x')dx'ds\le(a_2/a_1)^{d/2}\<w,\mu_0^*\>^2t.
\end{align*}
Similar estimates show that $\k_t^0(.,A)$ and $\k_t^-(.,A)$ are also bounded on $\R^d$.
Next, by differentiating under the integral sign and estimating similarly, we see that $\k_t^0(.,A)$
and $\k_t^\pm(.,A)$ are moreover Lipschitz on $\R^d$.

Recall that $E_n=\{y\in E:n^{-1}\le m(y)\le n\}$. Then $a$ and $a^{-1}$ are bounded on $E_n$.
Define, for $t\in(0,T)$ and $n\in\N$, signed kernels $\k_t^n$ on $\R^d\times\cE$ by
\begin{equation}\label{PM2}
\k_t^n(x,A)=\k_t^0(x,A\cap E_n)+\k_t^+(x,A\cap E_n)-\k_t^-(x,A\cap E_n).
\end{equation}
Then $\k_t^n(.,A)$ is bounded and Lipschitz on $\R^d$ for all $A\in\cE$ and, since $(\mu_t)_{t<T}$ is a solution, we have
$$
\mu_t(B\times(A\cap E_n))=\int_{x\in B}\k_t^n(x,A)dx,\q B\in\cB(\R^d).
$$
Hence $0\le\k_t^n(x,A)\le\k_t^{n+1}(x,A)$ for all $A\in\cE$ and $\<w,\k_t^n\>(x)\le\<w,\mu_0^*\>$ for all $x$.
Hence we can define $\k_t\in\cK$ for $t\in(0,T)$ by
$$
\k_t(x,A)=\lim_{n\to\infty}\k_t^n(x,A)
$$
and then $\k_t$ is a density for $\mu_t$ for all $t$.
Now, for all $x\in\R^d$ and $A\in\cE$, we have, as $n\to\infty$,
\begin{align*}
\k_t^0(x,A\cap E_n)
&=P_t\k_0(x,A\cap E_n)\to P_t\k_0(x,A)\\
\k_t^\pm(x,A\cap E_n)
&=\int_0^t(P_{t-s}K^\pm(\k_s))(x,A\cap E_n)ds\to\int_0^t(P_{t-s}K^\pm(\k_s))(x,A)ds.
\end{align*}
So, on rearranging \eqref{PM2} and letting $n\to\infty$, we see that $(\k_t)_{t<T}$ is a precise solution to \eqref{CD}.
Finally, if $(\k'_t)_{t<T}$ is any precise solution to \eqref{CD} which is a density for $(\mu_t)_{t<T}$,
then, for all $n\in\N$ and all $A\in\cE$ with $A\sse E_n$, the map $x\mapsto\k'_t(x,A)$ is Lipschitz on $\R^d$, by the argument above, 
so $\k'_t(x,A)=\k_t(x,A)$ for all $x$.
Hence $(\k_t)_{t<T}$ is unique.
\end{proof}

\section{Related work}\label{RW}
Prior work has considered function solutions, either in the discrete case, where $\mu_t(x,dy)=\sum_{m=1}^\infty f_t^m(x)\d_m(dy)$, or the continuous case, when $\mu_t(x,dy)=f_t(x,y)dy$.
On the question of existence in the discrete case, see \cite{MR2595194,MR1954870,MR1491848,MR1886116,MR2026914}.
We will restrict our review on existence to works addressing the continuous case.
Amann \cite{MR2001f:35195} proved local existence, uniqueness and mass conservation in a general setting, assuming uniform bounds on diffusivity and coagulation rates and uniform positivity of the diffusivity.
Later, Amann and Walker \cite{MR2174971}, proved global existence for small initial data under similar hypotheses.
Lauren\c{c}ot and Mischler \cite{MR1892231} proved global existence when the diffusivity $a:(0,\infty)\to(0,\infty)$ and its reciprocal are bounded on compacts 
and the coagulation kernel $k:(0,\infty)^2\to[0,\infty)$ satisfies the Galkin--Tupchiev monotonicity condition
\begin{equation}\label{TV}
k(y,y')\le k(y,y+y'),\q y,y'\in(0,\infty)
\end{equation}
along with the growth bounds
\begin{equation}\label{TW}
\sup_{y,y'\le R}k(y,y')<\infty,\q\sup_{y\le R}\frac{k(y,y')}{y'}\to0\q\text{as $y'\to\infty$},\q\text{for all $R$}.
\end{equation}
Both \cite{MR2001f:35195} and \cite{MR1892231} include a term modelling particle fragmentation, 
while \cite{MR2001f:35195} allows also for spatially dependent diffusion,
and \cite{MR2174971} allows for a further particle-shattering transition. 
None of these is possible in our model. 
See also Bailleul \cite{MR2827105} for an interesting special case of coagulation with spatially dependent diffusion.
Mischler and Rodriguez Ricard \cite{MR1979355} showed that the approach of \cite{MR1892231} can be extended 
(in the context of coagulation-diffusion in a bounded domain in $\R^3$) to the case where \eqref{TV} and \eqref{TW} are replaced by the weaker monotonicity condition
$$
k(y,y')\le k(y,y+y')+k(y',y+y'),\q y,y'\in(0,\infty)
$$
and growth bounds
\begin{equation}\label{TX}
\sup_{y,y'\in[R^{-1},R]}k(y,y')<\infty,\q\sup_{R^{-1}\le y\le R}\frac{k(y,y')}{y'}\to0\q\text{as $y'\to\infty$},\q\text{for all $R$}.
\end{equation}
We use a different approach to these papers, which allows us to dispense with any monotonicity condition but requires a different type of growth bound on $k$ for large particles.

As noted by Ball and Carr \cite{MR92h:82086} in the spatially homogeneous setting, the questions of uniqueness and mass conservation for coagulation equations are related to the existence of moment bounds for solutions. 
Here, the discrete and continuous cases do not differ substantially.
Hammond and Rezakhanlou \cite{MR2350433} and Rezakhanlou \cite{MR2726119} obtained suitable moment bounds for solutions under assumptions including that the diffusivity $a$
is positive, uniformly bounded and non-increasing, and that the coagulation kernel $k$ satisfies
$$
\sup_{y,y'}\frac{k(y,y')}{yy'}<\infty,\q \frac{k(y,y')}{(y+y')(a(y)+a(y'))}\to0\q\text{as $y+y'\to\infty$}.
$$
Rezakhanlou \cite{MR3187683} has shown that the non-increasing condition on the diffusivity can be relaxed to some extent.
We will retain this non-increasing condition but are able to prove uniqueness and mass conservation also when the diffusivity and coagulation kernels are unbounded for particles of small mass.

The approach taken in this paper is an extension to the spatially inhomogeneous setting of that developed in \cite{MR2002c:82079}. 
A version of Theorem \ref{T1} for the special case discussed in Section \ref{BC} is stated, along with a sketch of elements of the proof, in \cite[Section 3]{MR2249640}.

\section{Existence and uniqueness}\label{EU}
The notions of solution, strong solution and conservative solution for the coagulation-diffusion equation \eqref{CD} are defined in Section \ref{SOL}.
The notion of Markov solution is defined in Section \ref{OS}.

\begin{theorem}\label{T1}
Assume that the diffusivity $a$ and the coagulation kernel $K$ satisfy condition \eqref{A} and that $\mu_0\in\cM$ satisfies condition \eqref{M0}.
Set $\alpha =\l w^2, \mu^*_0 \r $.  
There exists $\zeta (\mu_0)\in [\alpha^{-1}, \infty]$ and a strong solution $(\mu_t)_{t < \zeta(\mu_{0})}$ to the coagulation-diffusion equation \eqref{CD} starting from $\mu_0$ with
the following property: if $(\nu_t)_{t<T}$ is any other solution to
\eqref{CD} starting from $\mu_0$, then
\begin{itemize}
\item[{\rm (i)}]
$T \leq \zeta (\mu_0)$ implies $\nu_t=\mu_t$ for all $t<T$,
\item[{\rm (ii)}] $T > \zeta (\mu_0)$ implies $(\nu_t)_{t<T}$ is not strong.
\end{itemize}
Moreover,
\begin{itemize}
\item[{\rm (iii)}] $(\mu_t)_{t<\z(\mu_0)}$ is a Markov solution to \eqref{CD}, 
\item[{\rm (iv)}] if 
%$mv^2\le w^3$ and both $(mw)^{1/2}$ and 
%$m^{1/2}w^{3/2}$ are $K$-subadditive and  both
$\|\<mw,\mu_0\>\|_1<\infty$, 
%and $\|\<m^{1/2}w^{3/2},\mu_0\>\|_1<\infty$, 
then $(\mu_t)_{t<\zeta(\mu_{0})}$ is conservative. 
\end{itemize}
Suppose further that, for some non-negative measurable function $v$ on $E$,
with $v/w$ bounded,
\begin{equation}\label{VW}
K(y,y', E) \leq w(y) v(y') + v(y) w (y') 
\end{equation}
then
\begin{itemize}
\item[{\rm (v)}] if $\a<\infty$ and $a^{-d/2}w v$ is $K$-subadditive, 
then $\zeta (\mu_0)=\infty$.
\end{itemize}
\end{theorem}

The proof of Theorem \ref{T1} will rely on an approximation scheme, the elements of which are constructed in the following two lemmas.
The proofs of the lemmas are given below.
Recall that $E_n = \{y \in E : n^{-1} \leq m(y) \leq n\}$.
\def\j{
and $\tilde E_{0,n}=E\sm E_n$. 
Define, for $m\ge0$, $\tilde E_{m+1,n}=E_n\times\tilde E_{m,n}$ and set
$$
\tilde E_n=\bigcup_{m\ge0}\tilde E_{m,n}.
$$
Fix $n$ and set $\tilde E=E_n\cup\tilde E_n$.
Write $\tilde\cE$ for the obvious $\s$-algebra on $\tilde E$, derived from $\cE$.
The function $w$ is already defined on $E=E_n\cup\tilde E_{0,n}$. 
Write $\tilde w$ for the unique extension of $w$ to $\tilde E$ such that $w(y,y')=w(y)+w(y')$ for all $y\in E_n$ and $y'\in\tilde E$.
Define a kernel $\tilde K$ on $\tilde E\times\tilde E\times\tilde\cE$ by
$$
\tilde K(y,y'dz)=1_{\{y,y'\in E_n\}}K(y,y',dz)
+w(y)w(y')1_{\{y\in E_n,y'\in\tilde E_n\}}\d_{(y,y')}(dz)
+w(y)w(y')1_{\{y\in \tilde E_n,y'\in E_n\}}\d_{(y',y)}(dz).
$$
GETS TOO COMPLICATED TO HANDLE DIFFUSION OF MULTIPARTICLES
}
Consider the coagulation kernels 
$$
K_n(y,y',dz)=1_{\{z\in E_n\}}K(y,y',dz),\q\tilde K_n(y,y',dz)=1_{\{z\in E\sm E_n\}}K(y,y',dz)
$$
and define $K^\pm_n (\mu)$ and $\tilde K^-_n(\mu)$ by analogy with \eqref{K+}, \eqref{K-}.  
Thus, in particular, $K^-(\mu)=K^-_n(\mu)+\tilde K_n^-(\mu)$.
Set $K_n(\mu)=K^+_n(\mu)-K^-_n(\mu)$.

\begin{lemma}\label{L1}
For each $n \in {\N}$, there exist processes $(\mu^n_t)_{t \geq 0}$ and $(\lambda^n_t)_{t \geq 0}$ in $\cM$ such that
$$
\mu^n_0 = 1_{E_{n}}\mu_0,\q \mu_t^n=1_{E_n}\mu_t^n,\q\lambda^n_0 = 1_{E^{c}_{n}} \mu_0
$$
and such that, setting $\eta^n_t = \l w, \lambda^n_t\r$, we have, for all $t\ge0$,
\begin{equation}\label{MBn}
\<w,\mu^n_t\>+\eta^n_t=\<w,P_t\mu_0\>+\int^t_0\<w,P_{t-s}K_n(\mu^n_s)\>ds\q\text{a.e.}
\end{equation}
and
\begin{align}\label{AE1}
\mu^n_t
&=P_t\mu^n_0+\int^t_0P_{t-s}(K^+_n(\mu^n_s)-K^-(\mu^n_s)-\eta^n_sw \mu^n_s)ds\\
\label{AE2}
\lambda^n_t
&=P_t\lambda^n_0+\int^t_0P_{t-s}(\tilde K^-_n(\mu^n_s)+\eta^n_sw \mu^n_s)ds.
\end{align}
Moreover $(\mu^n_t)_{t \geq 0}$ satisfies
$$
\mu_t^n=\tilde P^n_{t0}\mu^n_0+\int_0^t\tilde P^n_{ts}K^+(\mu^n_s)ds,\q t\ge 0
$$
where $(\tilde P^n_{ts}:0\le s\le t)$ is the propagator on measures on $\R^d\times E$ associated to the time-dependent operator $\frac12a(y)\Delta-c_t^n(.,y)$ on $\R^d$ and 
$$
c^n_t(x,y)=\int_{E\times E}K(y,y',E)\mu_t^n(x,dy)\mu_t^n(x,dy')+w(y)\eta_t^n(x).
$$
\end{lemma}

The final term in \eqref{MBn} is interpreted by analogy with
\eqref{F}, so is well-defined and non-positive.  The inequality
obtained by dropping this term
$$
\<w,\mu^n_t\>+\eta^n_t\le\<w,P_t\mu_0\>\q\text{a.e.}
$$
ensures that the integrals in equations \eqref{AE1}, \eqref{AE2} are well-defined as signed measures of finite total variation on $\R^d\times E$.
The new system of equations is designed so that $(\mu_t^n)_{t\ge0}$ should approximate the solution to \eqref{CD},
while $(\lambda_t^n)_{t\ge0}$ allows us to estimate the behaviour of particles of small or large mass.
The equations can be interpreted as follows: there are $\mu$-particles and $\lambda$-particles; $\mu$-particles all have type in $E_n$ and
coagulate as normal to produce new $\mu$-particles, except where the mass of the new particle would be greater than $n$, when it is
designated a $\lambda$-particle; $\lambda$-particles also act on $\mu$-particles, turning them into $\lambda$-particles. 
All particles diffuse at a speed determined by their type, as before.
\begin{lemma}\label{L2}
For all $n\in {\N}$ and all $t \geq 0$,
\begin{equation}\label{Z}
\mu^n_t\le\mu^{n+1}_t,\q\<w,\mu^n_t\>+\eta^n_t\ge\<w,\mu^{n+1}_t\>+\eta^{n+1}_t\q\text{a.e.} 
\end{equation}
Moreover, for all solutions $(\mu_t)_{t<T}$ to \eqref{CD} starting from $\mu_0$, for all
$n\in\N$ and $t<T$,
\begin{equation}\label{Y}
\mu^n_t\le\mu_t,\q\<w,\mu^n_t\>+\eta^n_t\ge\<w,\mu_t\>\q\text{a.e.}
\end{equation}
\end{lemma}

\begin{proof}[Proof of Theorem \ref{T1}]
Define a process $(\mu_t)_{t\ge0}$ in $\cM$ and a process of non-negative measurable functions $(\eta_t)_{t\ge0}$ by the monotone limits
$$
\mu_t=\lim_{n\to\infty}\mu^n_t,\q\eta_t=\lim_{n\to\infty}\eta^n_t\q\text{a.e.}
$$
Note that
$$
\<w,\mu_t\>+\eta_t\le\<w,P_t\mu_0\>\le\<w,\mu^*_0\>\q\text{a.e.}
$$
In particular, $\|\<w,\mu_t\>\|_1\le\|\<w,\mu_0\>\|_1<\infty$, so $(\mu_t)_{t\ge 0}$ satisfies \eqref{S1}.
Next
\begin{align*}
\|\int^t_0\<1,P_{t-s}K^-(\mu_s)\>ds\|_1
&=\int^t_0\|\<1,K^-(\mu_s)\>\|_1ds\\ 
&\le\int^t_0\|\<w,\mu_s\>^2\|_1ds
%\le\int^t_0 \|\l w, P_s \mu_0 \r \|_1 \, \| \l w, P_s \mu_0\r \|_\infty \, ds 
\le\|\<w,\mu_0\>\|_1\<w,\mu_0^*\>t<\infty.
\end{align*}
This estimate and other similar estimates allow us to use dominated convergence to pass to the limit in equation \eqref{AE1} to obtain
$$
\mu_t + \int^t_0 P_{t-s} (K^-(\mu_s) +\eta_s w \mu_s) ds = P_t\mu_0
+\int^t_0 P_{t-s} K^+ (\mu_s)\, ds.
$$
Hence if $\eta_t=0$ almost everywhere, for all $t<T$, then $(\mu_t)_{t<T}$ satisfies \eqref{WCD}.

Now suppose that $(\nu_t)_{t<T}$ is any solution to \eqref{CD} starting from $\mu_0$. 
By Lemma \ref{L2}, for $t<T$,
$$
\mu_t\le\nu_t,\q\<w,\mu_t\>+\eta_t\le\<w,\nu_t\>,\q\text{a.e.}
$$
Since $w$ is positive, if $\eta_t = 0$ almost everywhere, then $\nu_t = \mu_t$, for all $t<T$.  
In the case where $(\nu_t)_{t<T}$ is strong we have
$$
\int^t_0 \|\l w^2, \mu_s\r \|_\infty \, ds \leq \int^t_0 \|\l w^2, \nu_s\r \|_\infty ds < \infty, \quad t<T.
$$
So we can multiply \eqref{AE2} by $w$, integrate over $E$, and pass to the limit $n\to\infty$, using dominated convergence, to obtain
$$
\eta_t=\int^t_0\<w^2,P_{t-s}(\eta_s\mu_s)\>ds\q\text{a.e}.
$$
Then
$$
\|\eta_t\|_1 = \int^t_0 \| \l w^2, \eta_s \mu_s\r \|_1 ds \leq \int^t_0 \|\l w^2, \mu_s \r \|_\infty \| \eta_s \|_1 ds .
$$
Since $\|\eta_t\|_1$ is non-decreasing in $t$ and finite, this implies $\|\eta_t\|_1=0$, so $\eta_t =0$ almost everywhere, so $\nu_t=\mu_t$, for all $t<T$. 
Thus, while any strong solution persists, it is the only solution.  

We will now show that
\begin{equation}\label{SS}
\int^t_0\|\<w^2,\mu_s\>\|_\infty ds<\infty,\q t<\a^{-1},
\end{equation}
which by the preceding argument implies that $(\mu_t)_{t<\a^{-1}}$ is a strong solution.
Apply $P_s$ to equation \eqref{AE1}, multiply by $w^2$ and integrate over $E$ to obtain, for all $s,t \geq 0$ and all $n\in {\N}$, 
\begin{equation}\label{W2}
\l w^2, P_s \mu^n_t\r \leq \l w^2, P_{s+t}\mu^n_0 \r + \int^t_0 \l w^2, P_{s+t-r} K_n (\mu^n_r)\r dr\q\text{a.e.}
\end{equation}
Set $h_n(t) = \sup_{s\geq 0} \|\l w^2, P_s \mu^n_t \r\|_\infty$. 
For all $t>0$ and $x,x'\in \R^d$, for  $p(y)=p^{t,x,x'}(y) = p(a(y)t,x,x')$, both $w p$ and $w$ are $K$-subadditive.  
So, for $K(y,y',.)$-almost all $z$,
$$
w^2(z)p(z) - w^2(y)p(y)-w^2 (y') p(y')
\leq w(y) p(y) w(y') + w(y) p(y') w(y').
$$
Hence, for any $v$ satisfying \eqref{VW},
\begin{align*}
\l w^2,&P_{s+t-r} K_n (\mu^n_r)\r(x)\\
&\leq  \int_{\R^{d}\times E\times E} w(y) p^{s+t-r,x,x'} 
(y)w(y') K_n (y,y',E) \mu^n_r (x',dy) \mu^n_r (x',dy')\, dx'\\
&\leq  \|\l w^2p^{s+t-r,x,.}, \mu^n_r \r\l w v, \mu^n_r\r 
  + \l w v p^{s+t-r,x,.}, \mu^n_r\r\l w^2,\mu^n_r\r\|_1\\
&\leq  \|\l w v, \mu^n_r\r \|_\infty \l w^2, P_{s+t-r} \mu^n_r\r(x)
+  \|\l w^2, \mu^n_r\r \|_\infty \l w v, P_{s+t-r}\mu^n_r\r (x).
\end{align*}
Note that
$$
\sup_{s\geq 0} \|\l w^2, P_{s+t} \mu_0\r \|_\infty \leq \l w^2,
\mu^*_0\r=\alpha.
$$
On taking $v=w/2$ we obtain
$$
h_n(t) \leq \alpha + \int^t_0 h_n (r)^2 dr
$$
which implies $h_n(t) \leq (T-t)^{-1}$ for $t<T \equiv \alpha^{-1}$. This
bound is independent of $n$, so we obtain \eqref{SS} on letting $n\to\infty$.  

On the other hand, suppose that $0\leq v\leq Cw$ and that the function $a^{-d/2}w v$ is $K$-subadditive.
We can replace $w^2$ by $vw$ in \eqref{W2}. The final term is then 
non-positive, as at \eqref{F}, so we obtain
$$
\|\l w v, P_s \mu_t^n\r\|_\infty \leq \|\l w v, P_{s+t} \mu_0^n\r\|_\infty 
\leq C\|\l w^2, P_{s+t} \mu_0^n\r\|_\infty
\leq C\alpha.
$$ 
Hence
$$
h_n(t) \leq \alpha + 2C\alpha \int^t_0 h_n (r) \, dr .
$$
which implies $h_n(t) \leq \alpha e^{2C\alpha t}$ for all $t\geq0$.
This bound is independent of $n$, so, if $\a<\infty$, then $\zeta (\mu_0) = \infty$.

To prove (iii), note that, by Lemma \ref{L2}, $c^n_t(x,y)\da c_t(x,y)$ for almost all $x$, 
for all $y$ and $t<T$, so $\tilde P^n_{ts}\ua P^\mu_{ts}$ for all $s\le t<T$.
From Lemma \ref{L1}, we have
$$
\mu_t^n=\tilde P^n_{t0}\mu^n_0+\int_0^t\tilde P^n_{ts}K^+(\mu^n_s)ds,\q t\ge 0.
$$
So we can let $n\to\infty$ to obtain
$$
\mu_t=P^\mu_{t0}\mu_0+\int_0^tP^\mu_{ts}K^+(\mu_s)ds,\q t<T.
$$

It remains to prove (iv). 
Set $m_R(y) = m(y) 1_{m(y)\leq R}$. 
Then, for $t<\zeta(\mu_0)$, 
\begin{equation}\label{MC}
\|\l m_R, \mu_t \r\|_1 = \|\l m_R, \mu_0 \r \|_1 
+ \int^t_0 \int_{\R^{d} \times E \times E} l_R (y,y') K(y,y',E)
\mu_s (x,dy)\mu_s (x,dy')dxds 
\end{equation}
where 
$$
l_R(y,y') = (m(y) + m(y')) 1_{m(y)+m(y') \leq R} - m_R(y) - m_R(y'). 
$$
Now $l_R(y,y')\to 0$ as $R\to\infty$ and
$$
|l_R (y,y')| K(y,y',E)\leq (m(y)+ m(y'))w(y)w(y').
$$
So, if we can show 
\begin{equation}\label{MCE}
\int_0^t\|\<mw,\mu_s\>\<w,\mu_s\>\|_1ds<\infty
\end{equation}
then, by dominated convergence, the last term in \eqref{MC} tends to $0$ as
$R\to\infty$.  Hence $\|\l m, \mu_t\r\|_1 = \|\l m, \mu_0\r\|_1$ for
all $t<\zeta(\mu_0)$ as required.

Note that
$$
\|\<mw,\mu^n_t\>\|_1\le \|\<mw,\mu^n_0\>\|_1
+\int_0^t\int_{\R^d}\<mw,K_n(\mu_s^n)\>(x)dxds.
$$
Both $m$ and $w$ are $K$-subadditive, so for $K(y,y',.)$-almost all $z$,
$$
(mw)(z)-(mw)(y)-(mw)(y')\le m(y)w(y')+m(y')w(y).
$$
Hence
$$
\<mw,K_n(\mu_s^n)\>\le\<w^2,\mu_s^n\>\<mw,\mu_s^n\>.
$$
Set $k_n(t)=\|\<mw,\mu_t^n\>\|_1$. Then $k_n(0)=\|\<mw,\mu_0\>\|_1<\infty$,
$k_n(t)<\infty$ for all $n$ and $t$, and, for $s\le t<\zeta(\mu_0)$,
$$
k_n(t)\le k_n(0)+\int_0^s\|\<w^2,\mu_r\>\|_\infty k_n(r)dr.
$$
Hence, using Gronwall's lemma and then letting $n\to\infty$ we obtain
$$
\sup_{s\le t}\|\<mw,\mu_t\>\|_1<\infty
$$

\end{proof}

\begin{proof}[Proof of Lemma \ref{L1}]
It will be convenient to assume that $\l w,\mu^*_0\r=1$. The general case then follows by a scaling argument.

Consider the vector space $\cV_\infty$ of signed measures $\mu=\mu^+-\mu^-$ on $\R^d\times E$ such that $\mu^\pm\in\cM$ and
$$
\|\mu\|_\infty=\||\mu|(.,E)\|_\infty<\infty.
$$
Here $|\mu|=\mu^++\mu^-$ is the total variation measure of $\mu$, 
$(|\mu|(x,A):x\in\R^d,A\in\cE)$ is (a version of) the kernel for $|\mu|$ with respect to Lebesgue measure on $\R^d$, 
and $\|.\|_\infty$ (on the right) is the $L^\infty$-norm on measurable functions on $\R^d$.
Then $\cV_\infty$ is complete in the given norm.
Given $\mu\in\cV_\infty$ and a (suitably integrable) measurable function $f$ on $E$, 
we obtain another signed measure $f\mu$ on $\R^d\times E$ by multiplication. 
We write $\<f,\mu\>$ for the measurable function on $\R^d$ given by
$$
\<f,\mu\>(x)=\int_Ef(y)\mu(x,dy).
$$
Thus $\<f,\mu\>$ is determined only almost everywhere on $\R^d$, and we have
$$
(f\mu)(B\times E)=\int_B\<f,\mu\>(x)dx,\q
\|f\mu\|_\infty=\|\<|f|,|\mu|\>\|_\infty.
$$

Extend $K^\pm_n, K_n, \tilde{K}^-_n$ and $P_t$ to $\cV_\infty$ in the obvious way.  
Note that, for $\mu,\mu' \in \cV_\infty$ with $w\mu,w\mu'\in\cV_\infty$, we have $K^-(\mu),K^-(\mu')\in\cV_\infty$ and
$$
\|K^-_n(\mu)\|_\infty\le\|w\mu\|^2_\infty
$$
and
$$
\|K^-_n(\mu)-K^-_n(\mu')\|_\infty\le\|w(\mu-\mu')\|_\infty\| w(\mu+\mu')\|_\infty.
$$
Similar estimates hold for $K^+_n, K_n$ and $\tilde K^-_n$.  
Write $C$ for a finite constant, depending only on $a,w$ and $n$, whose value may vary from line to line.
Note that, if $\mu\in\cV_\infty$ is supported on $\R^d\times E_n$, then $w\mu\in\cV_\infty$ and
$$
\|w\mu\|_\infty\le C\|\mu\|_\infty,\q\|P_t\mu\|_\infty\le C\|\mu\|_\infty
$$
and moreover the signed measures $K^\pm_n(\mu)$ and $\tilde K^-_n(\mu)$ are also supported on $\R^d\times E_n$.

Set $\mu^n_0=1_{E_{n}} \mu_0$ and $\lambda^n_0=1_{E^{c}_{n}}\mu_0$. 
We can define a sequence of measurable maps 
$$
t\mapsto(\mu^{n,k}_t,\lambda^{n,k}_t):[0,\infty)\to\cV_\infty\times\cV_\infty,\q k\ge0
$$
with $\mu^{n,k}_t$ supported on $\R^d\times E_n$ for all $k$ and $t$, by setting $\mu^{n,0}_t=P_t\mu^n_0$ and $\lambda^{n,0}_t=P_t\lambda^n_0$,
and then recursively setting $\eta^{n,k}_t=\<w,\lambda^{n,k}_t\>$ and
\begin{align*}
\mu^{n,k+1}_t
&=P_t\mu^n_0+\int^t_0P_{t-s}(K^+_n(\mu^{n,k}_s)-K^-(\mu^{n,k}_s)-\eta^{n,k}_s w \mu^{n,k}_s)ds\\
\lambda^{n,k+1}_t
&=P_t\lambda^n_0+\int^t_0P_{t-s}(\tilde K^-_n(\mu^{n,k}_s)+\eta^{n,k}_sw\mu^{n,k}_s)ds.
\end{align*}
Fix $n$ and set 
$$
f_k(t)=\|w\mu^{n,k}_t\|_\infty+\|w\lambda^{n,k}_t\|_\infty.
$$
Then $f_0(t)=\|\l w,P_t\mu_0\r\|_\infty \leq \l w, \mu^*_0\r=1$ and, for $k \geq 0$,
$$
f_{k+1}(t)\le1+C\int^t_0f_k(s)^2ds,\q t\ge0.
$$
Hence $f_k(t) \leq  (1-Ct)^{-1}$ for $t<C^{-1}$. Set $T= (2C)^{-1}$ then
$f_k(t) \leq 2C$ for $t \leq T$.  Next, set $g_0(t)=f_0(t)$ and, for
$k \geq 0$,
$$
g_{k+1}(t) = \| w(\mu^{n,k+1}_t - \mu^{n,k}_t) \|_\infty + \| w 
(\lambda^{n,k+1}_t - \lambda^{n,k}_t)\|_\infty.
$$
Then, for $t \leq T$ and $k \geq 0$,
$$
g_{k+1}(t) \leq C \int^t_0 g_k (s)\, ds. 
$$
Hence, by a standard argument, $(w\mu^{n,k}_t, w\lambda^{n,k}_t)$ converges in $\cV_\infty \times \cV_\infty$ as $k\to \infty$, uniformly in $t \leq T$.  
The limit $(w\mu^n_t, w\lambda^n_t)_{t\le T}$ is a measurable map $[0,T]\to \cV_\infty \times \cV_\infty$, 
with $\mu^n_t$ supported on $\R^d\times E_n$ for all $t$, such that $\|w\mu_t^n\|_\infty+\|w\lambda_t^n\|_\infty\le2C$ and, setting $\eta^n_s=\<w,\lambda^n_s\>$,  
\begin{align}\label{BE1}
\mu^n_t 
&=P_t\mu^n_0+\int^t_0P_{t-s}(K^+_n(\mu^n_s)-K^-(\mu^n_s)-\eta^n_sw \mu^n_s)ds\\ 
\label{BE2}
\lambda^n_t
&=P_t\lambda^n_0+\int^t_0P_{t-s}(\tilde K^-_n(\mu^n_s)+\eta^n_sw\mu^n_s)ds.
\end{align}
On disintegrating these equations with respect to Lebesgue measure, adding them together, multiplying by $w$ and integrating over $E$, we obtain \eqref{MBn}.

For $u>0$, if we act first on \eqref{BE1} and \eqref{BE2} by $P_u$, we obtain instead
\begin{equation}\label{MBU}
\<w,P_u(\mu^n_t+\lambda^n_t)\>=\<w,P_{u+t}\mu_0\>+\int^t_0\<w,P_{u+t-s}K_n(\mu^n_s)\>ds.
\end{equation}
Assume for now that $\mu_t^n$ and $\lambda_t^n$ are non-negative for $t\le T$.
Then, the last term in \eqref{MBU} is non-positive, as in \eqref{F}. 
In particular, if $\tilde{\mu}^n_0 = \mu^n_T$ and $\tilde{\lambda}^n_0=\lambda^n_T$, then
$$
\tilde{f_0}(t):  = \| w P_t\tilde{\mu}^n_0 \|_\infty + 
\| w P_t \tilde{\lambda}^n_0 \|_\infty 
= \|\l w, P_t (\mu^n_T + \lambda^n_T)\r \|_\infty 
\leq \|\l w, P_{t+T} \mu_0\r\|_\infty \leq 1.
$$
The construction we have just made can therefore be applied with 
$\tilde{\mu}^n_0,
\tilde{\lambda}^n_0$ in place of $\mu^n_0, \lambda^n_0$ to obtain
$(\tilde{\mu}^n_t, \tilde{\lambda}^n_t)_{t \leq T}$. Then, setting
$\mu^n_{T+t} = \tilde{\mu}^n_t$ and $\lambda^n_{T+t}= \tilde{\lambda}^n_t$
for $t \leq T$, $(\mu^n_t, \lambda^n_t)_{t \leq 2T}$ satisfies \eqref{BE1},
\eqref{BE2}. By repeated extension we obtain a long-time solution.  

It remains to show that $\mu^n_t$ and $\lambda_t^n$ are non-negative for $t\le T$.  
Define $c^n_t$ and $\tilde P^n_{ts}$ for $s\le t\le T$ as in the statement.
Note that, for all $y\in E_n$ and all $t\le T$, we have
$$
|c^n_t(x,y)|\le w(y)\<w,|\mu^n_t|+|\lambda^n_t|\>(x)\le 2C\|w1_{E_{n}}\|_\infty,\q\text{a.e.}
$$
If $\nu$ is supported on $\R^d\times E_n$, then we have
$$
\|\tilde P^n_{ts}\nu\|_\infty\le C\|\nu\|_\infty.
$$
Set $\tilde{\mu}^{n,0}_t = \tilde{P}^n_{t0} \mu^n_0$ and define for $k \geq 0$
$$
\tilde{\mu}^{n,k+1}_t = \tilde{P}^n_{t0} \mu^n_0 + \int^t_0 \tilde{P}^n_{ts} K^+_n (\tilde{\mu}^{n,k}_s)ds.
$$
The arguments used above show, possibly for some smaller value of $T$, but independent of $\mu^n_0$ and $\lambda^n_0$, that $\tilde{\mu}^{n,k}_t$ converges in $\cV_\infty$, uniformly in $t \leq T$.  
The limit $(\tilde\mu^n_t)_{t\le T}$ is a process in $\cV_\infty$
with $\tilde\mu_t^n$ supported on $\R^d\times E_n$ for $t\le T$, such that
\begin{equation}\label{E}
\tilde\mu^n_t=\tilde P^n_{t0}\mu^n_0+\int_0^t\tilde P^n_{ts}K^+_n(\tilde\mu^n_s)ds,\q t\le T.
\end{equation}
By induction, we see that $\tilde\mu^{n,k}_t\ge0$ for all $k$, so $\tilde\mu_t^n\ge0$.  
Apply Proposition \ref{WSS} with $g_t=-c_t^n$ and $\a_t=K_n^+(\tilde\mu_t^n)$ to see that
\begin{equation}\label{E'}
\tilde\mu^n_t=P_t\mu^n_0+\int^t_0P_{t-s}(K^+_n(\tilde\mu^n_s)-c^n_s\tilde\mu^n_s)ds,\q t\le T. 
\end{equation}
Now $c^n_t\mu^n_t=K^-(\mu^n_t)+\eta^n_tw\mu^n_t$, so \eqref{E'} is also satisfied by $(\mu^n_t)_{t\le T}$.
The estimates we already have for $K^+_n$ and $c^n_t$ allow us to prove uniqueness for \eqref{E'}.
Hence $\mu^n_t=\tilde\mu_t^n\ge0$ for all $t\le T$.
A similar argument shows that $\lambda^n_t\ge0$ for all $t\le T$.  
\end{proof}
\begin{proof}[Proof of Lemma \ref{L2}]
Recall that, for all $n\in\N$, we have
\begin{equation}\label{AA}
\mu^n_t=P_t\mu^n_0+\int^t_0P_{t-s}(K^+_n(\mu^n_s)-c^n_s\mu^n_s)ds,\q t\ge0
\end{equation}
and
\begin{equation}\label{BZ}
\mu^{n+1}_t=P_t\mu^{n+1}_0+\int^t_0P_{t-s}(K^+_{n+1}(\mu^n_s)-c^{n+1}_s\mu^{n+1}_s)ds,\q t\ge0
\end{equation}
and
\begin{equation}\label{BY}
\mu_t=P_t\mu_0+\int^t_0P_{t-s}(K^+(\mu_s)-c_s\mu_s)ds,\q t<T 
\end{equation}
where
$$
c_s(x,y)=\int_EK(y,y',E)\mu_s(x,dy').
$$
Set $\pi^n_t=1_{E_n}\mu^{n+1}_t-\mu^n_t$ and $\pi_t=1_{E_n}\mu_t-\mu^n_t$, and note that $\pi^n_0 = \pi_0 = 0$.  
Set
\begin{equation}\label{CZ}
\rho^n_t=\int^t_0\<w,P_{t-s}(K_n(\mu^n_s)-K_{n+1}(\mu^{n+1}_s))\>ds,\q
\rho_t=\int^t_0\<w,P_{t-s}(K_n(\mu^n_s)-K(\mu_s))\>ds
\end{equation}
and set $\chi_t=\<w,\mu^n_t\>+\eta^n_t-\<w,\mu_t\>$.
By \eqref{MBn} and \eqref{B1} we have
$$
\rho^n_t=\<w,\mu^n_t\>+\eta^n_t-\<w,\mu^{n+1}_t\>-\eta^{n+1}_t,\q\rho_t\le\chi_t,\q\text{a.e}.
$$
Now, for all $y\in E$ and all $t$,
\begin{align*}
(c^n_t - c^{n+1}_t) (x,y) 
&=\int_E (w(y)w(y')-K(y,y',E)) \pi^n_t (x,dy) +w(y) \rho^n_t(x)+\gamma^n_t (x,y),\q\text{a.e.}\\
(c^n_t - c_t)(x,y) 
&=\int_E (w(y)w(y')-K(y,y',E)) \pi_t (x,dy) + w(y)\rho_t(x)+\gamma_t (x,y),\q\text{a.e.}
\end{align*}
where
\begin{align*}
\gamma^n_t(x,y)
&=\int_E(w(y)w(y')-K(y,y',E))1_{E\sm E_n}(y')\mu^{n+1}_t(x,dy')\\
\gamma_t(x,y)
&=\int_E(w(y)w(y')-K(y,y',E))1_{E\sm E_n}(y')\mu_t(x,dy')+w(y)(\chi_t-\rho_t)(x).
\end{align*}
Note that $0\le\gamma^n_t(x,y)\le Cw(y)$ and $0\le\gamma_t(x,y)\le Cw(y)$ for all $y\in E$, for almost all $x$. 
Multiply equations \eqref{BZ} and \eqref{BY} by $1_{E_n}$ and subtract equation \eqref{AA} to obtain
$$
\pi^n_t=\int^t_0P_{t-s}(A^n_s(\pi^n_s,\rho^n_s)-c_s^n\pi_s^n+\a^n_s)ds,\q\pi_t=\int^t_0P_{t-s}(A_s(\pi_s,\rho_s)-c_s^n\pi_s+\alpha_s)ds
$$
where, writing $K^+_n(.,.)$ for the polarization of $K^+_n$, for $\pi\in\cV_\infty$ and $\rho\in L^\infty(\R^d)$,
\begin{align*}
A^n_s(\pi,\rho)
&=K^+_n(\pi,\mu^{n+1}_s+\mu^n_s)+1_{E_n}\left(\int_E(w(\cdot)w(y')-K(\cdot,y',E))\pi(\cdot,dy')+w\rho\right)\mu_s^{n+1},\\
A_s(\pi,\rho)
&=K^+_n(\pi,\mu_s+\mu^n_s)+1_{E_n}\left(\int_E(w(\cdot)w(y')-K(\cdot,y',E))\pi(\cdot,dy')+w\rho\right)\mu_s
\end{align*}
and
$$
\a^n_s=K^+_n(1_{E\sm E_n}\mu_s^{n+1},\mu_s^{n+1}+\mu_s^n)+1_{E_n}\g_s^n\mu_s^{n+1},\q
\a_s=K^+_n(1_{E\sm E_n}\mu_s,\mu_s+\mu_s^n)+1_{E_n}\g_s\mu_s.
$$
Note that $\alpha^n_s \geq 0$, $\alpha_s \geq 0$ and $\|\alpha^n_s\|_\infty \leq C$,  $\|\alpha_s\|_\infty \leq C$.  
Also, if $\pi$ and $\rho$ are non-negative and $\pi$ is supported on $\R^d\times E_n$, then
$A^n_s(\pi,\rho)\geq 0$ and $A_s(\pi,\rho) \geq 0$, and
$$
\|A^n_s(\pi,\rho)\|_\infty\le C(\|\pi\|_\infty+\|\rho\|_\infty),\q
\|A_s(\pi,\rho)\|_\infty\le C(\|\pi\|_\infty+\|\rho\|_\infty).
$$
We can rewrite the equations \eqref{CZ} in the form
$$
\rho^n_t =  \int^t_0 \l w, P_{t-s} (B^n_s (\pi^n_s) + \beta^n_s)\r ds, \q
\rho_t =  \int^t_0 \l w, P_{t-s} (B_s (\pi_s) + \beta_s)\r ds, 
$$
where
$$
B^n_s (\pi) = - K_n(\pi, \mu^n_s + \mu^{n+1}_s),\quad B_s (\pi) = -K_n
(\pi, \mu^n_s +\mu_s)
$$
and
\begin{align*}
\beta^n_s
&=-(K_{n+1}-K_n)(\mu^{n+1}_s)-K_n(1_{E\sm E_n}\mu^{n+1}_s,\mu^n_s+\mu^{n+1}_s),\\
\beta_s
&=-(K-K_n)(\mu_s)-K_n(1_{E\sm E_n}\mu_s,\mu^n_s+\mu_s).
\end{align*}
Since $\tilde w^{t-s,x,x'}$ is $K$-subadditive (see \eqref{F}), we have
$$
\l w, P_{t-s} \beta^n_s\r \geq 0, \quad \l w, P_{t-s} \beta_s \r \geq 0
$$
and, for $\pi \geq 0$,
$$
\l w, P_{t-s} B^n_s (\pi)\r \geq 0, 
\quad \l w, P_{t-s} B_s (\pi) \r \geq 0 .
$$
Note that
$$
\|\<w,P_{t-s}\b^n_s\>\|_\infty\le C,\q\|\<w,P_{t-s}(\b_s+K(\mu_s))\>\|_\infty\le C
$$
and
$$
\|\int^t_0 \l w, P_{t-s} K(\mu_s)\r ds\|_\infty \leq \| \l w, P_t \mu_0 \r \|_\infty \leq C.
$$
From this point on we pursue only the proof of \eqref{Y}.  
The argument for \eqref{Z} is the same.  
Define, recursively for $k\ge0$, a process of measures $(\tilde\pi_t^k)_{t<T}$ supported on $\R^d\times E_n$ 
and a process of non-negative measurable functions $(\tilde\rho_t^k)_{t<T}$ on $\R^d$
by setting $\tilde{\pi}^0_t=0$, $\tilde{\rho}^0_t=0$ and then
$$
\tilde\pi^{k+1}_t=\int^t_0\tilde P_{ts}^n(A_s(\tilde\pi^k_s,\tilde\rho^k_s)+\alpha_s)ds,\q
\tilde\rho^{k+1}_t=\int^t_0\<w,P_{t-s}(B_s(\tilde\pi^k_s)+\beta_s)\>ds. 
$$
The estimates obtained allow us to show that $(\tilde{\pi}^k_t, \tilde{\rho}^k_t)_{t<T}$ converges in $\cV_\infty\times L^\infty(\R^d)$, 
uniformly on compacts in $t$, and  that the limit $(\tilde\pi_t,\tilde\rho_t)_{t<T}$ is a measurable map
$$
t\mapsto(\tilde\pi_t,\tilde\rho_t):[0,T)\to\cV_\infty\times L^\infty(\R^d)
$$
with  $\tilde\pi_t$ supported on $E_n$ for all $t$, such that
\begin{align}
\tilde\pi_t\notag
&=\int^t_0\tilde P_{ts}^n(A_s(\tilde\pi_s,\tilde\rho_s)+\alpha_s)ds, \\
\tilde\rho_t\label{JF}
&=\int^t_0\<w,P_{t-s}(B_s(\tilde\pi_s)+\beta_s)\>ds. 
\end{align}
Apply Proposition \ref{WSS} with $g_t=-c_t^n$ and `$\a_t$'$=A_t(\tilde\pi_t,\tilde\rho_t)+\a_t$ to see that
\begin{equation}\label{JE}
\tilde\pi_t=\int^t_0P_{t-s}(A_s(\tilde\pi_s,\tilde\rho_s)-c_s^n\tilde\pi_s+\alpha_s)ds.
\end{equation}
But $(\pi_t,\rho_t)_{t<T}$ also satisfies the equations \eqref{JF},\eqref{JE} and our estimates allow us to 
show these equations have only one solution.
Hence $\pi_t =\tilde{\pi}_t \geq 0$ and $\rho_t = \tilde{\rho}_t \geq 0$ almost everywhere, for all $t<T$, which implies \eqref{Y}.
\end{proof}

We finish this section with a result which allows us to recover function solutions in the case $E=(0,\infty)$,
whenever the initial mass distribution is either discrete or absolutely continuous.
Say that a measure $\mu$ on $E$ has {\em integer mass distribution} if it is supported on the set $m^{-1}(\N)$.
Say that $\mu$ has {\em absolutely continuous mass distribution} if there is a kernel $\rho$ on $(0,\infty)\times\cE$ 
such that, $m(y)=m$ for $\rho(m,\cdot)$-almost all $y$, for all $m\in(0,\infty)$, and such that
$$
\mu(dy)=\int_{(0,\infty)}\rho(m,dy)dm
$$
where $dm$ denotes Lebesgue measure.
Extend these notions in the obvious way to measures on $\R^d\times E$ and to processes of such measures.

Given a suitable process $(\mu_t)_{t<T}$ in $\cM$, we defined the propagator $(P_{ts}^\mu:0\le s\le t<T)$ just before \eqref{MS}.
Consider now the equation 
\begin{equation}\label{NUE}
\nu_t=P_{t0}^\mu\mu_0+\int_0^t P_{ts}^\mu K^+(\nu_s)ds,\q t<T
\end{equation}
for a process $(\nu_t)_{t<T}$ in $\cM$.
This has a minimal non-negative solution, given by $\nu_t=\lim_{k\to\infty}\nu_t^k$, where
we set $\nu_t^0=P_{t0}^\mu\mu_0$ and define recursively, for $k\ge0$,
$$
\nu_t^{k+1}=P_{t0}^\mu\mu_0+\int_0^t P_{ts}^\mu K^+(\nu_s^k)ds,\q t<T.
$$

\begin{theorem}\label{IFU}
Assume that the diffusivity $a$ and the coagulation kernel $K$ satisfy condition \eqref{A} and that the initial measure $\mu_0$ satisfies condition \eqref{M0}.
Set $T=\z(\mu_0)$ and let $(\mu_t)_{t<T}$ be the maximal strong solution to \eqref{CD} starting from $\mu_0$, as in Theorem \ref{T1}. 
Then $(\mu_t)_{t<T}$ is also the minimal non-negative solution to \eqref{NUE}.
Moreover, if $\mu_0$ has integer mass distribution, then so does $(\mu_t)_{t<T}$. 
Further, if $\mu_0$ has absolutely continuous mass distribution, then so does $(\mu_t)_{t<T}$.
\end{theorem}
\begin{proof}
We know that $(\mu_t)_{t<T}$ satisfies \eqref{NUE} by Theorem \ref{T1} (iii). 
We have $\nu_t^0=P_{t0}^\mu\mu_0\le\mu_t$ for all $t$.
By induction, $\nu_t^k\le\mu_t$ for all $k$, so $\nu_t\le\mu_t$ for all $t$.
Fix $n$ and define an increasing sequence of processes $(\nu^{n,k}_t)_{t\ge0}$ in $\cM$ by setting $\nu_t^{n,0}=\tilde P^n_{t0}\mu_0^n$ for all $t$ and then for $k\ge0$
$$
\nu_t^{n,k+1}=\tilde P^n_{t0}\mu_0^n+\int_0^t\tilde P^n_{ts}K^+(\nu_s^{n,k})ds,\q t\ge0.
$$
Set $\nu_t^{n,\infty}=\lim_{k\to\infty}\nu_t^{k,n}$. 
Then $(\nu_t^{n,\infty})_{t\ge0}$ satisfies \eqref{E}, so $\nu_t^{n,\infty}=\mu_t^n$ for all $n$ and $t$.
Now, for $0\le s<t<T$, we have $\tilde P_{ts}^n\le P_{ts}^\mu$ and $\mu_t^n\ua\mu_t$.
Hence $\nu^{n,k}_t\le\nu_t^k$ for all $n$ and $k$ and so $\nu_t^k\ua\mu_t$ as $k\to\infty$ for all $t<T$.

It is easy to see by induction that, if $\mu_0$ has integer mass distribution, 
then this property is inherited by $(\nu_t^k)_{t<T}$ for all $k$, and hence by $(\mu_t)_{t<T}$. 
The same holds for the property of absolutely continuous mass distribution.  
\end{proof}

\section{Einstein--Smoluchowski coagulation}\label{BC}
Einstein \cite{AE1905} derived a formula for the diffusivity $D$ (here meaning $a/2$) of a spherical particle of radius $r$, suspended in a fluid of viscosity $\eta$
$$
D=\frac{kT}{6\pi\eta r}
$$
where $k$ is Boltzmann's constant and $T$ is the temperature.
Smoluchowski \cite{S} considered a population of such particles, subject to coagulation on collision. 
He argued that this would result in a rate for coagulation between particles of radius $r$ and $r_*$ given by
$$
4\pi(D+D_*)(r+r_*).
$$
Here, of course $d=3$. 
In our notation, after choosing suitable units, this leads to the case
$$
a(y)=y^{-1/3},\q K(y,y_*)=(y^{-1/3}+y^{-1/3}_*)(y^{1/3}+y^{1/3}_*).
$$
Hammond and Rezakhanlou \cite{MR2308858} (also Yaghouti, Rezakhanlou and Hammond \cite{MR2554038}) derived the Smoluchowski kernel rigorously, 
starting from a system Brownian particles, moving independently, 
except for random coalescence events on the Boltzmann--Grad scale.
This is a case then where the coagulation-diffusion equation \eqref{CD} is of some interest.
Mischler and Rodriguez Ricard's existence result \cite{MR1979355} applies here, for the variant model with diffusion reflected in a smoothly bounded domain.
Uniqueness holds by Rezakhanlou \cite{MR2726119}, provided we impose a positive lower cut-off on particle mass.

We can recover existence, uniqueness and mass conservation, in the whole space $\R^3$, without lower mass cut-off.
To see this, take $E=(0,\infty)$ and consider the sublinear function $\phi(y)=y^{1/6}+y^{5/6}$.
Then $w(y)=a(y)^{d/2}\phi(y)=y^{-1/3}+y^{1/3}\ge1$ and $K(y,y')\le w(y)w(y')$, so condition \eqref{A} holds. 
Moreover, if we take $v(y)=2y^{-1/3}$, then $v/w$ is bounded, $a^{-3/2}vw$ is subadditive, and condition \eqref{VW} holds.
So we can apply Theorem \ref{T1}, provided we impose on the initial measure $\mu_0\in\cM$ the conditions $\mu_0(dx,dy)\le dx\otimes\mu_0^*(dy)$ and
$$
\int_0^\infty\int_{\R^3}(y^{-1/3}+y^{1/3})\mu_0(dx,dy)<\infty,\q
\int_0^\infty(y^{-2/3}+y^{2/3})\mu_0^*(dy)<\infty
$$
for some measure $\mu_0^*$.
Then there is a unique (global) strong solution to the coagulation-diffusion equation \eqref{CD}.
Moreover, if further
$$
\int_0^\infty\int_{\R^3}y^{4/3}\mu_0(dx,dy)<\infty
$$
then the total mass $\int_0^\infty\int_{\R^3}y\mu_t(dx,dy)$ is conserved for all $t$.

\bibliography{p}
\end{document}